\documentclass[11pt,notitlepage,a4paper,reqno]{amsart}
\usepackage[usenames]{color}
\usepackage{colortbl}

\newcommand{\medint}{-\kern  -,375cm\int}

\theoremstyle{plain}
\newtheorem{theorem}{Theorem}[section]

\newtheorem{lemma}[theorem]{Lemma}

\theoremstyle{definition}

\theoremstyle{remark}

\theoremstyle{plain}

\numberwithin{equation}{section} \makeatletter

\renewcommand{\p@enumi}{\thesection.}
\makeatother  \allowdisplaybreaks

\title[Existence and regularity results of weak solutions]{Existence and regularity results for weak solutions to $(p,q)$-elliptic systems in divergence form}
\author[ M.~Bul\'{\i}\v{c}ek]{Miroslav Bul\'{\i}\v{c}ek}
\address{Charles University, Faculty of Mathematics and Physics, Mathematical Institute, Sokolovsk\'{a} 83, 186~75, Prague, Czech Republic}
\email{mbul8060@karlin.mff.cuni.cz}

\author[G.~Cupini]{Giovanni Cupini}
\address{Dipartimento di Matematica,
  Universit\`{a} di Bologna, Piazza di Porta S.Donato 5, 40126  Bologna, Italy }
\email{giovanni.cupini@unibo.it}

\author[B.~Stroffolini]{Bianca Stroffolini}
\address{Dipartimento di Matematica e Applicazioni, Universit\`{a} di Napoli ''Federico II", via Cintia, 80126 Napoli, Italy}
\email{bstroffo@unina.it}

\author[A.~Verde]{Anna Verde}
\address{Dipartimento di Matematica e Applicazioni, Universit\`{a} di Napoli "Federico II", via Cintia, 80126 Napoli, Italy}
\email{anverde@unina.it}


\thanks{M.~Bul\'{\i}\v{c}ek's work was supported by the ERC-CZ project LL1202 financed by the Ministry of Education, Youth and Sports, Czech Republic. M.~Bul\'{\i}\v{c}ek
is a member of the Ne\v{c}as center for Mathematical Modeling.   The other authors are members of the Gruppo Nazionale per l'Analisi
Matematica, la Probabilit\`a e le loro Applicazioni (GNAMPA) of the
Istituto Nazionale di Alta Matematica (INdAM)}

\begin{document}

\begin{abstract}
We prove existence and regularity results for weak solutions of non linear elliptic systems with non variational structure satisfying
$(p,q)$-growth conditions. In particular we are able to prove higher differentiability results under a dimension-free gap between $p$ and $q$.
\end{abstract}

\maketitle



\section{Introduction}\label{S:1}

In this paper we focus on the existence and the regularity results for solutions $u$ to the  Dirichlet
problems associated with the following  nonlinear system  in divergence form (here $\alpha=1,\ldots,N$)
\begin{equation}\label{dirichlet}
\left\{\begin{array}{ll}
\displaystyle \sum_{i=1}^n \frac{\partial}{\partial x_i} A_{i}^{\alpha}(Du)=0 &\text{in $ \Omega$}
\\
u=u_0 &\text{on $ \partial \Omega$,}
\end{array}\right.
\end{equation}
where  the functions $A_i^{\alpha}(\xi)$  are locally Lipschitz continuous in $\mathbb{R}^{nN}$,  $\Omega$ is an open bounded subset of
$\mathbb{R}^n$ and $Du:\Omega \to \mathbb{R}^{nN}$ represents the gradient of a (vector-valued) function $u:\Omega \to \mathbb{R}^N$.

We equip the problem with the general $(p,q)$-growth conditions, i.e., we assume that there are $1< p\le q<\infty$ and  two positive constants
$m, M$ such that for all $\xi, \lambda\in \mathbb{R}^{nN}$  and for all  $i,j=1,\ldots,n$, and  $\alpha,\beta=1,\ldots,N$ there holds
\begin{align}\label{ellitticita1}
m(1+|\xi|^2)^{\frac{p-2}{2}}|\lambda|^2&\le
\sum_{i,j=1}^n\sum_{\alpha,\beta=1}^N \frac{\partial A_{i}^{\alpha}}{\partial \xi^{\beta}_j}(\xi)\lambda^{\alpha}_i\lambda_j^{\beta},\\
\label{crescita-q}
 \left|\frac{\partial A_{i}^{\alpha}}{\partial \xi_j^{\beta}}(\xi)\right|&\le  M (1+|\xi|^2)^{\frac{q-2}{2}}.
\end{align}
Notice that \eqref{ellitticita1} is the usual ellipticity condition and  \eqref{crescita-q} is the $q$-growth condition,
from which the name of $(p,q)$-growth come from.  Under these assumptions, one can easily observe (see Lemma \ref{stimaLp}) that $|A^{\alpha}_i(\xi)|\le C(1+|\xi|)^{q-1}$ with some generic constant $C$ and therefore we can naturally define a notion of a weak solution to \eqref{dirichlet} in the following way:

Let $u_0\in W^{1,p}(\Omega; \mathbb{R}^N)\cap W^{1,q}_{\rm loc}(\Omega;\mathbb{R}^N)$. We say that $u$ is a weak solution to \eqref{dirichlet} if
\begin{equation}\label{soldebole-Dirichlet1}
 u-u_0 \in W_0^{1,p} (\Omega;\mathbb{R}^N)\cap W_{\textrm{loc}}^{1,q} (\Omega;\mathbb{R}^N)
\end{equation}
and for all open $\Omega'$ fulfilling $\overline{\Omega'}\subset \Omega$ and for all $\varphi \in W_0^{1,q}(\Omega';\mathbb{R}^N)$ there holds
\begin{equation}\label{soldebole-Dirichlet}
\int_{\Omega}\sum_{i=1}^{n}\sum_{\alpha=1}^{N}A_{i}^{\alpha}(Du)\varphi^{\alpha}_{x_i}(x) \,dx=0.
\end{equation}
Here, and also in what follows, we use the abbreviation $\varphi^{\alpha}_{x_i}:=\frac{\partial \varphi^{\alpha}}{\partial x_i}$

Our main task in the paper is to establish the existence of such a solution and further some regularity of arbitrary weak  solutions.
However, contrary to the classical result, we do not in general assume any symmetry condition on the derivative of $A^{\alpha}_i$ and so we do not assume that the system is in variational form.
Nevertheless, as done in \cite{mar91}  in  the scalar framework,
we will need to compensate this lack of symmetry by the following assumption on the asymptotic behavior of the skew-symmetric part, namely,
 for all $\xi, \lambda\in \mathbb{R}^{nN}$  and for all  $i,j=1,\ldots,n$, and $\alpha,\beta=1,\ldots,N$ there holds
\begin{equation}\label{continuita}
\left|\frac{\partial A_{i}^{\alpha}}{\partial \xi_j^{\beta}}(\xi)-\frac{\partial A_{j}^{\beta}}{\partial \xi_i^{\alpha}} (\xi)\right|\le M
(1+|\xi|^2)^{\frac{q+p-4}{4}}.
\end{equation}

If $p = q$, the existence of weak solutions to (\ref{dirichlet}) can be established using the theory of coercive,
monotone operators, see Leray--Lions \cite{lerlio65}, Browder \cite{browder} and Hartman--Stampacchia \cite{{hart-stamp}}. Also the regularity issue has been extensively studied, see the
monographs \cite{gia2}, \cite{giusti} and the surveys \cite{Mingione} and \cite{mingione2}. Notice also, that without any further additional structural assumptions, the best\footnote{This information can be as usual slightly improved by the Gehring lemma} known regularity information about
the solution is that $V(Du)\in W^{1,2}_{\rm loc}(\Omega; \mathbb{R}^{nN})$, where
\begin{equation}\label{defV}
V(\xi):=(1+|\xi|^2)^{\frac{p-2}{4}}\xi.
\end{equation}
On the other hand, if $p < q$ the above classical existence results cannot be applied due to the lack of coercivity
in $W^{ 1,q}$ . Moreover, the request $u\in W^{1,q }_{\rm loc}(\Omega; \mathbb{R}^N)$ in the definition of weak solution, needed to
have a well defined integral, is an additional difficulty. Notice that such a request is a
priori assumed in some regularity results under the $p, q$-growth, see for example  \cite{leonetti1}, \cite{bilfuchsCalcVar} and  \cite{cupmarmas4}.

The first result of the paper is that any weak solution is in fact twice weakly differentiable.
\begin{theorem}\label{T:main-weak}
Let $1<p\le q<\infty$ be arbitrary and  $A$ satisfy \eqref{ellitticita1}, \eqref{crescita-q} and \eqref{continuita}.
Then any $u\in W^{1,\max\{q,2\}}_{\rm loc}(\Omega; \mathbb{R}^N)$ fulfilling \eqref{soldebole-Dirichlet} satisfies
for all $\eta\in C_{c}^{\infty}(\Omega)$ the following estimate
\begin{equation}
\int_{\Omega}(1+|Du|^2)^{\frac{p-2}{2}}
 |D^2u|^2\eta^{2}\, dx
\le c\int_{\Omega}
(1+|Du|^2)^{\frac{q}{2}} |D\eta|^2 \,dx,
\label{exdopoH3-1}
\end{equation}
where the constant $c$ depends only on $m$ and $M$. In particular, we also have that
\begin{equation}
\int_{\Omega}|D V(Du)|^2\eta^{2}\, dx
\le c\int_{\Omega}
(1+|Du|^2)^{\frac{q}{2}} |D\eta|^2 \,dx.
\label{exdopoH3-V}
\end{equation}
\end{theorem}

The above theorem provides the existence of the second derivatives for arbitrary $1<p\le q<\infty$ but the right hand side of \eqref{exdopoH3-1}
or \eqref{exdopoH3-V} still depends on the $W^{1,q}$ norm of $u$. We shall improve this estimate provided that $p$ and $q$
are sufficiently close to each other. Thus, the second main theorem of the paper is the following.
\begin{theorem}
\label{t:main}
Let $1<p\le q<\infty$ be arbitrary and  $A$ satisfy \eqref{ellitticita1}, \eqref{crescita-q} and \eqref{continuita} and
$u \in W^{1,\max\{q,2\}}_{\rm loc}(\Omega;\mathbb{R}^N)$ satisfy \eqref{soldebole-Dirichlet}.
Then for all open $\Omega' \subset \overline{\Omega'} \subset \Omega$ the following holds:
\begin{itemize}
\item[i)] If
 \begin{equation}\label{ipotesip-q}
q<p\frac{n+2}{n}
\end{equation}
then
$$
\int_{\Omega'}\left(
|V(Du)|^{\frac{2q}{p}}+|D V(Du)|^2 + (1+|D u|^2)^{\frac{p-2}{2}}|D^2u|^2\right)\, dx \le C(\Omega', n, N,  p, q,  m, M, \|Du\|_{L^p(\Omega)}).
$$
\item[ii)]
If $u\in L^{\infty}(\Omega;\mathbb{R}^N)$ and
 \begin{equation}\label{ipotesip-qb}
q<p+2 \qquad \textrm{and} \qquad p<n
\end{equation}
then
\begin{equation*}
 \begin{split}
\int_{\Omega'}&\left(|V(Du)|^{\frac{2q}{p}}+|D V(Du)|^2 + (1+|D u|^2)^{\frac{p-2}{2}}|D^2u|^2\right)\, dx
\\ &\qquad \qquad\qquad \qquad\qquad \qquad\qquad \qquad \le C(\Omega', n, N, p, q, m, M, \|Du\|_{L^p(\Omega)},\|u\|_{L^{\infty}(\Omega)}).
   \end{split}
\end{equation*}
%
\end{itemize}
In particular, in both cases we have that $V(Du)\in W^{1,2}_{\rm loc}(\Omega;\mathbb{R}^{nN})$,
which, due to the embedding theorem,  leads to $Du\in  L^{\frac{p2^*}{2}}_{\rm loc}(\Omega;\mathbb{R}^{nN})$.
\end{theorem}

Finally, we state our last main result of the paper. It is an  existence result
 for the Dirichlet problem  \eqref{dirichlet}. For this purpose, we need
 to consider a regularity assumption on the boundary  datum.
We shall require in what follows that
\begin{equation}\label{ipotesi-dato}
  u_0 \in W^{1,r} (\Omega;\mathbb{R}^N), \quad \textrm{with}\,\, r:=\max\left\{2,\frac{p(q-1)}{p-1}\right\}.
\end{equation}
\begin{theorem}
\label{t:main2}
Let $1<p\le q<\infty$ be arbitrary and  $A$ satisfy \eqref{ellitticita1}, \eqref{crescita-q}
and \eqref{continuita}. Moreover, let $u_0$ fulfill \eqref{ipotesi-dato}. Then there exists a weak solution to the problem \eqref{dirichlet} provided that at least one of the following conditions hold
\begin{itemize}
\item[i)] $p$ and $q$ satisfy \eqref{ipotesip-q}.
\item[ii)] $p$ and $q$ satisfy \eqref{ipotesip-qb},  $u_0\in L^{\infty}(\partial \Omega;\mathbb{R}^N)$ {and
\begin{equation}\label{structure}
  \sum_{i=1}^{n}A_{i}^{\alpha}(\xi)\xi_i^{\alpha}\ge 0, \ \  \forall\xi\in \mathbb{R}^{nN},\  \ \forall \alpha\in \{1,\ldots,N\}.
\end{equation}}
\end{itemize}
\end{theorem}

As far as the regularity of solutions is concerned, the obstructions are essentially two: we are
dealing with systems and under non-standard growth $(p < q)$. Indeed, in the vectorial
case, even under the standard growth, the everywhere regularity of solutions for systems, or of
minimizers of integrals, cannot be expected unless some structure conditions are assigned, and this
holds also for the local boundedness, see e.g. the counterexamples by De Giorgi \cite{deg}  and \v{S}ver\'ak-Yan
\cite{sverak}.
Since the pioneering paper by Marcellini \cite{mar89}, the theory of regularity in the framework of non-
standard growth has been deeply investigated. The results and the contributions to regularity are
so many, that it is a hard task to provide a comprehensive overview of the issue. For this, we refer
to the survey of Mingione \cite{Mingione} for an accurate and interesting account on this subject. A common
feature is that to get regularity results $p$ and $q$ must be not too far apart, as examples of irregular
solutions by Giaquinta \cite{gia}, Marcellini \cite{mar87} and Hong \cite{hong} show. On the other hand, many regularity results are available if
the ratio $q/p$ is bounded above by suitable constant that in general depends on the dimension $n$, and converges to $1$ when
$n$ tends to infinity (\cite{bil03}, \cite{CarKriPas}, \cite {espoleomin1}, \cite {espoleomin2}, \cite {espoleomin3}).
Moreover, the condition on the distance between the exponents
$p$ and $q$ can usually be relaxed if the solutions/minimizers are assumed locally bounded.

Let us observe that the
local higher differentiability results for bounded minimizers of  integral functionals satisfying $p, q$-growth conditions is more studied
than the analogous issue for systems of PDE's.
In particular, recently, the Authors, in \cite{CKP},  considered integral functionals
with convex integrand satisfying $p, q$-growth conditions. They proved local higher differentiability results for bounded minimizers
under dimension-free conditions on the gap between the growth and the coercivity exponents; i.e., \eqref{ipotesip-qb} restricted to the case $p\geq 2$,
using an improved Gagliardo-Nirenberg's inequality. We  also observe that  an existence result  in the $(p,q)$-framework was proved in  \cite{CupLeoMas}  for  a Dirichlet problem \eqref{dirichlet} with monotone  operators possibly  depending on the $x$-variable, but for $p\ge 2$  only. As a novel feature, the main results are achieved through uniform higher differentiability estimates for solutions to a class of
auxiliary problems, constructed adding higher order perturbations to the integrand.
Here we achieve the same result for systems with non variational structure with control on the skew-symmetric part (see (\ref{continuita})).

\medskip
The plan of the paper is the following. In Section \ref{s:preliminaries} we prove some
preliminar algebraic  inequalities. In Sections \ref{apriori} and \ref{P-t:main} we prove the higher differentiability results  Theorem \ref{T:main-weak} and
Theorem \ref{t:main}, respectively. In the last section, we prove the existence result (Theorem \ref{t:main}) for the  problem \eqref{dirichlet}.



\section{Auxiliary algebraic inequalities}\label{s:preliminaries}

In this part, we recall several algebraic inequalities related to the mapping $A$. Although, their proof can be in some simplified setting found in many works, see e.g. \cite[Lemma~4.4, Lemma 2.4]{mar91}, \cite[Lemma~1]{Tolksdorf}, \cite[Lemma 5.1]{cupmarmas4} or \cite[Chapter~5]{mnrr96}, we provide for the sake of clarity  a detailed proof here.
We start with the first auxiliary result based on the assumptions \eqref{ellitticita1}--\eqref{crescita-q}.
\begin{lemma} \label{stimaLp}
Let $A:\mathbb{R}^{nN}\to \mathbb{R}^{nN}$ be a continuous mapping fulfilling \eqref{ellitticita1} and \eqref{crescita-q}. Then there exists a positive constant $K$ such that for all $\xi, \eta \in \mathbb{R}^{nN}$
there hold
\begin{align}\label{dis-ellitticita}
 |\xi|^p &\leq  K\left\lbrace (1+|\eta|^2)^{\frac{p(q-1)}{2(p-1)}} + \sum_{i=1}^{n} \sum_{\alpha=1}^NA_i^{\alpha}(\xi)(\xi^{\alpha}_i-\eta_i^{\alpha})\right\rbrace,\\
 |A^{\alpha}_i(\xi)| &\leq K (1+|\xi|^2)^{\frac{q-1}{2}} \qquad \textrm{ for all } \alpha=1,\ldots, N \textrm{ and } i=1,\ldots,n. \label{crescitaai}\\
 \label{casop>2}
|\xi-\eta|^p&\le K\sum_{i=1}^{n} \sum_{\alpha=1}^{N}\left( A_i^{\alpha}(\xi)-A_i^{\alpha}(\eta)\right)  (\xi_i^{\alpha}-\eta_i^{\alpha}) \quad \textrm{for }p\ge 2,\\
\label{casop<2}
\left( 1+|\xi|^2+|\eta|^2\right)^{\frac{p-2}{2}}|\xi-\eta|^2  &\le K\sum_{i=1}^{n} \sum_{\alpha=1}^{N}\left(A_i^{\alpha}(\xi)-A_i^{\alpha}(\eta)\right)  (\xi^{\alpha}_i-\eta^{\beta}_i) \quad \textrm{for } p \in (1,2).
\end{align}
\end{lemma}

\begin{proof}
We start the proof with \eqref{crescitaai}. Since
\begin{equation}\label{Nap1}
A^{\alpha}_i(\xi)-A^{\alpha}_i(0)=\int_0^1 \sum_{j=1}^n \sum_{\beta=1}^N\frac{\partial A^{\alpha}_i(t\xi)}{\partial \xi^{\beta}_j}\xi^{\beta}_j\; dt,
\end{equation}
we can use the assumption \eqref{crescita-q}, to get
$$
\left|A^{\alpha}_i(\xi) \right|\le \left|A^{\alpha}_i(0)\right| + M \int_0^1 \sum_{j=1}^n \sum_{\beta=1}^N(1+t^2|\xi|^2)^{\frac{q-2}{2}}
|\xi^{\beta}_j|\; dt\le \left|A^{\alpha}_i(0)\right| + M nN\int_0^1 (1+t^2|\xi|^2)^{\frac{q-2}{2}} |\xi|\; dt.
$$
Thus, in case $q\ge 2$, the inequality \eqref{crescitaai} immediately follows.

If $q\in (1,2)$ we can continue with estimating the last integral in the following way
$$
\int_0^1 (1+t^2|\xi|^2)^{\frac{q-2}{2}} |\xi|\; dt = \int_0^{|\xi|} (1+t^2)^{\frac{q-2}{2}} \; dt \le 2^{\frac{2-q}{2}}\int_0^{|\xi|} (1+t)^{q-2}
\; dt \le \frac{2^{\frac{2-q}{2}}}{q-1}(1+|\xi|)^{q-1}
$$
and we again see that \eqref{crescitaai} follows directly.

To show \eqref{casop>2}--\eqref{casop<2}, we write
$$
\begin{aligned}
&\sum_{i=1}^{n}\sum_{\alpha=1}^{N}\left( A_i^{\alpha}(\xi)-A_i^{\alpha}(\eta)\right)  (\xi_i^{\alpha}-\eta_i^{\alpha}) =\int_0^1 \sum_{i,j=1}^{n}\sum_{\alpha=1}^{N} \frac{\partial A^{\alpha}_i(t\xi + (1-t)\eta)}{\partial \xi^{\beta}_j} ( \xi_j^{\beta}-\eta_j^{\beta})  (\xi_i^{\alpha}-\eta_i^{\alpha})\; dt\\
&\overset{\eqref{ellitticita1}}\ge m|\xi-\eta|^2\int_0^1 (1+|t\xi+(1-t)\eta|^2)^{\frac{p-2}{2}}\; dt.
\end{aligned}
$$
Then, following step by step proof of Lemma~1.19 in \cite[Chapter~5]{mnrr96}, we deduce  \eqref{casop>2}--\eqref{casop<2}.

To  show \eqref{dis-ellitticita}, we first consider the case $p \geq 2$. Then by using \eqref{casop>2} and \eqref{crescitaai} and also the Young's inequality, we can observe that for all  $\epsilon >0$ and all  $\xi,\eta \in \mathbb{R}^{nN}$, we have
\begin{align*}
&|\xi|^p \leq c(|\xi- \eta|^p +|\eta|^p) \leq c\left\{
\sum_{i=1}^{n}\sum_{\alpha=1}^{N}  (A_i^{\alpha}(\xi)-A_i^{\alpha}(\eta)) (\xi_i^{\alpha}-\eta_i^{\alpha})+|\eta|^p\right\}
\\ &\le
 c \left\lbrace |\eta|^p+  \sum_{i=1}^{n}\sum_{\alpha=1}^{N}  A_i^{\alpha}(\xi) (\xi_i^{\alpha}-\eta_i^{\alpha})+
\bar{C}(1+|\eta|^2)^{\frac{q-1}{2}}(|\xi|+|\eta|) \right\rbrace \\& \le
 c \left\lbrace (1+|\eta|^2)^{\frac{p}{2}}+\sum_{i=1}^{n} \sum_{\alpha=1}^{N}   A_i^{\alpha}(\xi) (\xi_i^{\alpha}-\eta_i^{\alpha})+
c_{\epsilon}(1+|\eta|^2)^{\frac{p(q-1)}{2(p-1)}}+\epsilon  (|\xi|+|\eta|)^p \right\rbrace;
 \end{align*}
thus if $\epsilon$ is  small enough we get  \eqref{dis-ellitticita}.

In the case $1<p<2$, we proceed slightly differently. By using the Young 's  inequality with complementary exponents $\frac{2}{p}$ and $ \frac{2}{2-p}$ we get  for $\epsilon >0$
\begin{align*}
|\xi|^p &\leq c\left(|\xi- \eta|^p +|\eta|^p\right) \le
 c\left( |\eta|^p+ (|\xi- \eta|^2)^{\frac{p}{2}}
 (1+|\xi|^2 +|\eta|^2)^{ \frac{p(p-2)}{4} + \frac{p(2-p)}{4}}\right)
\\ &\le
c \left\{(1+|\eta|^2)^{\frac{p}{2}}+ c_{\epsilon}
(1+ |\xi|^2+|\eta|^2)^{\frac{p-2}{2}}|\xi-\eta|^2+\epsilon (1+ |\xi|^2+|\eta|^2)^{\frac{p}{2}}\right\}.
\end{align*}
Therefore, by \eqref{casop<2}, with a proper choice of (small)  $\epsilon>0$,  we get
\[
|\xi|^p \leq c \left\{ (1+|\eta|^2)^{\frac{p}{2}} +\sum_{i=1}^{n} \sum_{\alpha=1}^{N}
(A_i^{\alpha}(\xi)-A_i^{\alpha}(\eta)) (\xi_i^{\alpha}-\eta_i^{\alpha})
 \right\}
\]
and we conclude by proceeding as above.
\end{proof}

The following estimate will play a crucial role for getting the information about the second derivatives of the weak solutions to \eqref{soldebole-Dirichlet}.
\begin{lemma}\label{Tolsk}
Let $A$ be a continuous mapping fulfilling \eqref{ellitticita1}, \eqref{crescita-q} and \eqref{continuita}. Then there exists a positive constant $K$ such that for all $\xi,\eta,\zeta\in \mathbb{R}^{nN}$ we have
\begin{equation}\label{key-3}
\begin{aligned}
\frac{m}{2}(1+|\zeta|^2)^{\frac{p-2}{2}}|\xi|^2 & \le \sum_{i,j=1}^n \sum_{\alpha,\beta=1}^N \frac{\partial A^{\alpha}_i(\zeta)}{\partial \zeta^{\beta}_j}(\xi^{\alpha}_i-\eta^{\alpha}_i)\xi^{\beta}_j+K(1+|\zeta|^2)^{\frac{q-2}{2}}|\eta|^2.
\end{aligned}
\end{equation}
\end{lemma}
\begin{proof}
For arbitrary $\zeta, \xi, \eta \in \mathbb{N}$, we define a bilinear form (for fixed $\zeta$)
$$
(\xi,\eta)_{\zeta}:=\frac12\sum_{i,j=1}^n \sum_{\alpha,\beta=1}^N\left(\frac{\partial A^{\alpha}_i(\zeta)}{\partial \zeta^{\beta}_j}+\frac{\partial A^{\beta}_j(\zeta)}{\partial \zeta^{\alpha}_i}\right)\eta^{\alpha}_i\xi^{\beta}_j.
$$
 Trivially  $(\xi,\eta)_{\zeta}=(\xi,\eta)_{\zeta}$.
Moreover, using  the assumption \eqref{ellitticita1} we get that
$$
(\xi,\xi)_{\zeta}=\sum_{i,j=1}^n \sum_{\alpha,\beta=1}^N\frac{\partial A^{\alpha}_i(\zeta)}{\partial \zeta^{\beta}_j}\xi^{\alpha}_i\xi^{\beta}_j\ge m(1+|\zeta|^2)^{\frac{p-2}{2}}|\xi|^2,
$$
and consequently, we see that for any fixed $\zeta$, the relation $(\xi,\eta)_{\zeta}$ is a scalar product on $\mathbb{R}^{nN}$ and therefore the Cauchy--Schwarz inequality holds, i.e.,
\begin{equation}
|(\xi,\eta)_{\zeta}|\le (\xi,\xi)^{\frac12}_{\zeta}(\eta,\eta)^{\frac12}_{\zeta}. \label{CSc}
\end{equation}
Thus, by  assumption \eqref{ellitticita1} and taking into account that
\[\sum_{i,j=1}^n \sum_{\alpha,\beta=1}^N \left(\frac{\partial A^{\alpha}_i(\zeta)}{\partial \zeta^{\beta}_j}-
\frac{\partial A^{\beta}_j(\zeta)}{\partial \zeta^{\alpha}_i}\right)\xi^{\alpha}_i\xi^{\beta}_j=0,\]
 we have
$$
\begin{aligned}
m&(1+|\zeta|^2)^{\frac{p-2}{2}}|\xi|^2 \le (\xi,\xi)_{\zeta} = -(\xi-\eta,\xi-\eta)_{\zeta}+ 2(\xi,\xi-\eta)_{\zeta} + (\eta,\eta)_{\zeta}\le
2(\xi,\xi-\eta)_{\zeta} + (\eta,\eta)_{\zeta}\\
&=\sum_{i,j=1}^n \sum_{\alpha,\beta=1}^N \left(\frac{\partial A^{\alpha}_i(\zeta)}{\partial \zeta^{\beta}_j}+
\frac{\partial A^{\beta}_j(\zeta)}{\partial \zeta^{\alpha}_i}\right)(\xi^{\alpha}_i-\eta^{\alpha}_i)\xi^{\beta}_j+
\sum_{i,j=1}^n \sum_{\alpha,\beta=1}^N \frac{\partial A^{\alpha}_i(\zeta)}{\partial \zeta^{\beta}_j}\eta^{\alpha}_i\eta^{\beta}_j\\
&=2\sum_{i,j=1}^n \sum_{\alpha,\beta=1}^N \frac{\partial A^{\alpha}_i(\zeta)}{\partial \zeta^{\beta}_j}(\xi^{\alpha}_i-
\eta^{\alpha}_i)\xi^{\beta}_j-\sum_{i,j=1}^n \sum_{\alpha,\beta=1}^N \left(\frac{\partial A^{\alpha}_i(\zeta)}{\partial \zeta^{\beta}_j}-
\frac{\partial A^{\beta}_j(\zeta)}{\partial \zeta^{\alpha}_i}\right)(\xi^{\alpha}_i-\eta^{\alpha}_i)\xi^{\beta}_j\\
&\qquad +\sum_{i,j=1}^n \sum_{\alpha,\beta=1}^N \frac{\partial A^{\alpha}_i(\zeta)}{\partial \zeta^{\beta}_j}\eta^{\alpha}_i \eta^{\beta}_j\\
&=2\sum_{i,j=1}^n \sum_{\alpha,\beta=1}^N \frac{\partial A^{\alpha}_i(\zeta)}{\partial \zeta^{\beta}_j}(\xi^{\alpha}_i-\eta^{\alpha}_i)\xi^{\beta}_j +
\sum_{i,j=1}^n \sum_{\alpha,\beta=1}^N \left(\frac{\partial A^{\alpha}_i(\zeta)}{\partial \zeta^{\beta}_j}-\frac{\partial A^{\beta}_j(\zeta)}{\partial \zeta^{\alpha}_i}
\right)\eta^{\alpha}_i\xi^{\beta}_j\\
&\qquad +\sum_{i,j=1}^n \sum_{\alpha,\beta=1}^N \frac{\partial A^{\alpha}_i(\zeta)}{\partial \zeta^{\beta}_j}\eta^{\alpha}_i \eta^{\beta}_j\\
&\overset{\eqref{crescita-q},\eqref{continuita}}\le 2\sum_{i,j=1}^n \sum_{\alpha,\beta=1}^N
\frac{\partial A^{\alpha}_i(\zeta)}{\partial \zeta^{\beta}_j}(\xi^{\alpha}_i-\eta^{\alpha}_i)\xi^{\beta}_j
+MnN(1+|\zeta|^2)^{\frac{q+p-4}{4}}|\eta||\xi|+MnN(1+|\zeta|^2)^{\frac{q-2}{2}}|\eta|^2.
\end{aligned}
$$
 Taking into account that
\[(1+|\zeta|^2)^{\frac{q+p-4}{4}}|\eta||\xi|=\left((1+|\zeta|^2)^{\frac{p-2}{4}}|\xi|\right)
\left((1+|\zeta|^2)^{\frac{q-2}{4}}|\eta|\right)\]
and using  the Young's inequality, we get
\[m(1+|\zeta|^2)^{\frac{p-2}{2}}|\xi|^2
\le 2\sum_{i,j=1}^n \sum_{\alpha,\beta=1}^N \frac{\partial A^{\alpha}_i(\zeta)}{\partial \zeta^{\beta}_j}
(\xi^{\alpha}_i-\eta^{\alpha}_i)\xi^{\beta}_j+\frac{m}{2}(1+|\zeta|^2)^{\frac{p-2}{2}}|\xi|^2+C(1+|\zeta|^2)^{\frac{q-2}{2}}|\eta|^2\] for a suitable constant $C$.
 Then, \eqref{key-3} easily follows.
\end{proof}

\section{Proof of Theorem~\ref{T:main-weak}}\label{apriori}

We proceed via difference quotients technique. Due to the assumed regularity of the solution $u$ and thanks to \eqref{crescitaai}, it follows from \eqref{soldebole-Dirichlet} that
$$
\int_{\Omega}\sum_{i=1}^n \sum_{\alpha=1}^N (A^{\alpha}_i(Du(x+he_k))-A^{\alpha}_i(Du(x)))\varphi^{\alpha}_{x_i}\; dx =0,
$$
for all $\varphi\in W^{1,q}_0(\Omega_h;\mathbb{R}^N)$, all $h\in (0,1)$ and all $k=1,\ldots, n$, where $\Omega_h:=\{x\in \Omega: \; B_{2h}(x)\subset \Omega\}$ and $e_k$ is a unit vector in the $k$-th direction. Hence, setting
$$
\varphi(x):=(u(x+he_k)-u(x))\tau^2(x)
$$
with $\tau \in \mathcal{C}^{\infty}_c(\Omega_{2h})$ (which is an admissible choice), we obtain the starting identity
\begin{equation}\label{start}
\begin{split}
0&=\int_{\Omega}\sum_{i=1}^n \sum_{\alpha=1}^N (A^{\alpha}_i(Du(x+he_k))-A^{\alpha}_i(Du(x)))\tau(x)\cdot\\
&\qquad \cdot \left((u^{\alpha}_{x_i}(x+he_k)-u^{\alpha}_{x_i}(x))\tau(x)+2(u^{\alpha}(x+he_k)-u^{\alpha}(x))\tau_{x_i}
\right)\; dx.
\end{split}
\end{equation}
Since
$$
\begin{aligned}
&A^{\alpha}_i(Du(x+he_k))-A^{\alpha}_i(Du(x))\\
&=\int_0^t\sum_{j=1}^n \sum_{\beta=1}^N \int_0^1 \frac{\partial A^{\alpha}_i(tDu(x+he_k)+(1-t)Du(x))}{\partial \zeta^{\beta}_j}(u^{\beta}_{x_j}(x+he_k)-u^{\beta}_{x_j}(x))\; dt,
\end{aligned}
$$
the identity \eqref{start} can be equivalently rewritten as
\begin{equation}\label{start2}
\begin{split}
0&=\int_{\Omega}\sum_{i,j=1}^n \sum_{\alpha,\beta=1}^N\int_0^1 \frac{\partial A^{\alpha}_i(tDu(x+he_k)+(1-t)Du(x))}{\partial \zeta^{\beta}_j}(u^{\beta}_{x_j}(x+he_k)-u^{\beta}_{x_j}(x))\tau(x)\cdot\\
&\qquad \cdot \left((u^{\alpha}_{x_i}(x+he_k)-u^{\alpha}_{x_i}(x))\tau(x)+2(u^{\alpha}(x+he_k)-u^{\alpha}(x))\tau_{x_i}
\right)\; dt\; dx.
\end{split}
\end{equation}
Abbreviating for the moment
$$
\xi^{\alpha}_i:= \tau(x)(u^{\alpha}_{x_i}(x+he_k)-u^{\alpha}_{x_i}(x)), \quad \eta^{\alpha}_i:=-2(u^{\alpha}(x+he_k)-u^{\alpha}(x))\tau_{x_i}(x)
$$
and
$$
\zeta:=tDu(x+he_k)+(1-t)Du(x),
$$
we can formally rewrite \eqref{start2} as
\begin{equation*}
\begin{split}
0&=\int_{\Omega}\int_0^1\sum_{i,j=1}^n \sum_{\alpha,\beta=1}^N \frac{\partial A^{\alpha}_i(\zeta)}{\partial \zeta^{\beta}_j}\xi^{\beta}_j\left(\xi^{\alpha}_i-\eta^{\alpha}_i\right)\; dt\; dx.
\end{split}
\end{equation*}
Thus, using \eqref{key-3}, we obtain (here $C$ is some constant depending only on $m,M,n,N,p,q$)
$$
\int_{\Omega}\int_0^1 (1+|\zeta|^2)^{\frac{p-2}{2}}|\xi|^2\; dt\; dx \le C\int_{\Omega}\int_0^1(1+|\zeta|^2)^{\frac{q-2}{2}}|\eta|^2\; dt \; dx,
$$
which in terms of original variables after division by $h^2$ means that
\begin{equation}
\begin{split}\label{huz}
&\int_{\Omega}\int_0^1 (1+|tDu(x+he_k)+(1-t)Du(x))|^2)^{\frac{p-2}{2}}\frac{|Du(x+he_k)-Du(x)|^2}{h^2}\tau^2(x)\; dt\; dx \\
&\le 4C\int_{\Omega}\int_0^1(1+|tDu(x+he_k)+(1-t)Du(x))|^2)^{\frac{q-2}{2}}\frac{|u(x+he_k)-u(x)|^2}{h^2}|D\tau(x)|^2\; dt \; dx.
\end{split}
\end{equation}

Finally, we let $h\to 0_+$. First, we focus on the limit in the term on the right hand side of \eqref{huz}.
In case that  $q\le 2$, we use the assumption that $u\in W^{1,2}_{\rm loc}(\Omega; \mathbb{R}^N)$ and therefore, we can use
the Lebesgue dominated convergence theorem to conclude that
$$
\begin{aligned}
\limsup_{h\to 0}&\int_{\Omega}\int_0^1(1+|tDu(x+he_k)+(1-t)Du(x))|^2)^{\frac{q-2}{2}}\frac{|u(x+he_k)-u(x)|^2}{h^2}|D\tau(x)|^2\; dt \; dx\\
&=\int_{\Omega}(1+|Du|^2)^{\frac{q-2}{2}}|u_{x_k}|^2|D\tau|^2 \; dx\le \int_{\Omega}(1+|Du|^2)^{\frac{q}{2}}|D\tau|^2 \; dx.
\end{aligned}
$$
Next, if $q>2$, we use the H\"{o}lder inequality, the assumption $u\in W^{1,q}_{\rm loc}(\Omega; \mathbb{R}^N)$
and the Lebesgue dominated convergence theorem to conclude
$$
\begin{aligned}
\limsup_{h\to 0}&\int_{\Omega}\int_0^1(1+|tDu(x+he_k)+(1-t)Du(x))|^2)^{\frac{q-2}{2}}\frac{|u(x+he_k)-u(x)|^2}{h^2}|D\tau(x)|^2\; dt \; dx\\
&=
\limsup_{h\to 0}\int_{\Omega}\int_0^1\left(((1+|tDu(x+he_k)+(1-t)Du(x))|^2)^{\frac{q-2}{2}}|D\tau(x)|^{2\frac{q-2}{q}}\right)\cdot \\
&\qquad
\cdot\frac{|u(x+he_k)-u(x)|^2}{h^2}|D\tau(x)|^{\frac{4}{q}}\; dt \; dx\\
&\le
\limsup_{h\to 0}\int_0^1 \left(\int_{\Omega}((1+|tDu(x+he_k)+(1-t)Du(x))|^2)^{\frac{q}{2}}|D\tau(x)|^{2}\; dx \right)^{\frac{q-2}{q}}\cdot \\
&\qquad
\left(\int_{\Omega}\frac{|u(x+he_k)-u(x)|^q}{h^q}|D\tau(x)|^{2} \; dx\right)^{\frac{2}{q}}\; dt\\
&\le \int_{\Omega}(1+|Du|^2)^{\frac{q}{2}}|D\tau|^2 \; dx.
\end{aligned}
$$

Consequently, substituting  these limits into \eqref{huz}, we have
\begin{equation}
\begin{split}\label{huz2}
&\limsup_{h\to 0}\int_{\Omega}\int_0^1 (1+|tDu(x+he_k)+(1-t)Du(x))|^2)^{\frac{p-2}{2}}\frac{|Du(x+he_k)-Du(x)|^2}{h^2}\tau^2(x)\; dt\; dx \\
&\le 4C\int_{\Omega}(1+|Du|^2)^{\frac{q}{2}}|D\tau|^2 \; dx.
\end{split}
\end{equation}
>From this estimate it immediately follows that $u\in W^{2,\min\{2,p\}}_{\rm loc}(\Omega; \mathbb{R}^N)$, in particular we know that $D^2u$
exists and that for almost all $x$
 $$
\frac{Du(x+he_k)-Du(x)}{h}\to (D^2 u)_{x_k}(x)
$$
where $(D^2 u)_{x_k}$ stands for  $\frac{\partial Du}{\partial x_k}$.
Therefore, we can use the Fatou lemma in \eqref{huz2} to conclude
\begin{equation*}
\begin{split}
&\int_{\Omega}(1+|Du|^2)^{\frac{p-2}{2}}|  D^2 u_{x_k}(x)|^2\tau^2\; dx\le 4C\int_{\Omega}(1+|Du|^2)^{\frac{q}{2}}|D\tau|^2 \; dx.
\end{split}
\end{equation*}
Since $k$ is arbitrary, the relation \eqref{exdopoH3-1} obviously follows. In addition, using the following algebraic inequality
$$
|DV(Du)|^2\le K(1+|Du|^2)^{\frac{p-2}{2}}|D^2u|^2,
$$
we see that \eqref{exdopoH3-V} holds as well. Hence the proof is complete.

\section{Proof of Theorem~\ref{t:main}}\label{P-t:main}

We shall start by recalling the definition of the Sobolev embedding exponent
\begin{equation}\label{2star} 2^{*}=\left\{\begin{array}{ll}\frac{2n}{n-2}&\text{if $n\ge 3$}\\
\text{arbitrary $>2$} & \text{if $n=2$}.\end{array}\right.\end{equation}
The value $2^*$ in dimension $n=2$ will be finally chosen sufficiently large. Since $u$ is assumed to be a weak solution belonging to $W^{1,\max\{q,2\}}_{\rm loc}(\Omega;\mathbb{R}^N)$, we can use Theorem~\ref{T:main-weak} and after summing \eqref{exdopoH3-V} and \eqref{exdopoH3-1}, we obtain the starting inequality valid for all $\tau \in \mathcal{C}^{\infty}_c (\Omega)$
\begin{equation}\label{split}
\int_{\Omega}\left((1+|Du|^2)^{\frac{p-2}{2}}|D^2u|^2\tau^2 + |DV(Du)|^2\tau^2\right) \; dx \le K\int_{\Omega}(1+|Du|^2)^{\frac{q}{2}}|D\tau|^2\; dx
\end{equation}

 Moreover, we remark that
\begin{equation}\label{e:Lq}
 \int_{\Omega}(1+|Du|^2)^{\frac{q}{2}}|D\tau|^2\; dx\le c\int_{\Omega}\left( (1+|Du|^2)^{\frac{p}{2}}+|V(Du)|^{\frac{2q}{p}}\right)|D\tau|^2\,dx.
\end{equation}
Indeed, in  $\{|Du|\le 1\}$ we have \[(1+|Du|^2)^{\frac{q}{2}}\le 2(1+|Du|^2)^{\frac{p}{2}}\]
and, in $\{|Du|> 1\}$,
\[(1+|Du|^2)^{\frac{q}{2}}= \left\{(1+|Du|^2)^{\frac{p-2}{2}}(1+|Du|^2)\right\}^{\frac{q}{p}}\le 2^{\frac{q}{p}}
|V(Du)|^{\frac{2q}{p}}.\]

Next, we split the proof for the case i) and ii).

\subsection{The case $q<p\frac{n+2}{n}$}
In this case, we first use the Sobolev embedding to conclude that (with some $C$ depending on $2^*$)
$$
\begin{aligned}
\|V(Du)\tau\|_{2^*}^2 &\le C\|D(V(Du)\tau)\|_2^2 \le 2C\int_{\Omega} \left(|DV(Du)|^2\tau^2 + |V(Du)|^2|D\tau|^2\right)\; dx \\
&\le 2C\int_{\Omega}\left(|DV(Du)|^2\tau^2 + (1+|Du|^2)^{\frac{p}{2}}|D\tau|^2\right)\; dx.
\end{aligned}
$$
Using this inequality in \eqref{split},
 and taking into account \eqref{e:Lq}
we get
\begin{equation}\label{split2}
\begin{split}
&\|V(Du)\tau\|_{2^*}^2+ \int_{\Omega}\left((1+|Du|^2)^{\frac{p-2}{2}}|D^2u|^2\tau^2 + |DV(Du)|^2\tau^2\right) \; dx \\
&\quad\le K_1\int_{\Omega}\left((1+|Du|^2)^{\frac{p}{2}}|D\tau|^2+(1+|Du|^2)^{\frac{q}{2}}|D\tau|^2\right)\; dx\\
&\quad\le K_2\int_{\Omega}\left((1+|Du|^2)^{\frac{p}{2}}|D\tau|^2+|V(Du)|^{\frac{2q}{p}}|D\tau|^2\right)\; dx.
\end{split}
\end{equation}
In particular, we have that $V(Du)\in L_{\rm loc}^{2^*}$.

Let us now estimate the last integral on the right hand side. Since $q\in (p,p\frac{2^*}{2})$, which follows from the assumption that $q<p \frac{n+2}{n}$ (note here that the value of $2^*$ in dimension $n=2$ has to be chosen greater   than $\frac{2q}{p}$), there exists a unique
$\theta\in (0,1)$ such that
$$
\frac{q}{2}=\frac{p}{2}(1-\theta)+\frac{p2^*}{4}\theta, \qquad \theta:=\frac{q-p}{p(\frac{2^*}{2}-1)}.
$$

As we will prove below,  under our assumptions on the exponents $p$ and $q$ and, if $n=2$,  with   a suitable choice of  $2^*$,  we have
\begin{equation}
2>2^*\theta. \label{asss}
\end{equation}
 Consider $\eta\in \mathcal{C}^{\infty}_c(\Omega)$ an arbitrary nonnegative cut-off function  and set
$$
\tau:=\eta^{\gamma} \qquad \textrm{with } \gamma:=\frac {2}{ 2-2^*\theta}.
$$
We have that
\begin{equation}\label{eta-tau}
\frac{|D\tau|^2}{\tau^{2^*\theta}}= \gamma^2\eta^{\gamma(2-2^*\theta)-2}|D \eta|^2= \gamma^2|D\eta|^2.
\end{equation}
Then by  the H\"older inequality, we have
\begin{equation*}
\begin{aligned}
\int_{\Omega}
|V(Du)|^{\frac{2q}{p}}|D\tau|^2\,dx&=\int_{\Omega}
|V(Du)|^{2(1-\theta)}(V(Du)\tau)^{2^*\theta} \frac{|D\tau|^2}{\tau^{2^*\theta}}\,dx\\
&\le \|V(Du)\|_2^{2(1-\theta)}\|V(Du)\tau\|_{2^*}^{2^*\theta} \left\|\frac{|D\tau|^2}{\tau^{2^*\theta}}\right\|_{\infty}
\end{aligned}
\end{equation*}
and we can apply the Young 's inequality to deduce that for arbitrary $\varepsilon>0$ we have
\begin{equation}\label{ssop}
\begin{aligned}
\int_{\Omega}
|V(Du)|^{\frac{2q}{p}}|D\tau|^2\,dx&\le \varepsilon \|V(Du)\tau\|_{2^*}^2 + C(\varepsilon,\gamma) \|V(Du)\|_2^{2(1-\theta)\gamma}
\left\|\frac{|D\tau|^2}{\tau^{2^*\theta}}\right\|^{\gamma}_{\infty}.
\end{aligned}
\end{equation}
Therefore, combining \eqref{split2}, \eqref{ssop}, \eqref{eta-tau}  and taking into account that
$V(Du)\le (1+|Du|^2)^{\frac{p}{4}}$, with a proper choice of $\varepsilon>0$, we obtain
\begin{equation*}
\begin{split}
&\|V(Du)\tau\|_{2^*}^2+\int_{\Omega}\left((1+|Du|^2)^{\frac{p-2}{2}}|D^2u|^2\eta^{2\gamma} + |DV(Du)|^2\eta^{2\gamma}\right) \; dx \\
&\quad\le K(\gamma,\|D \eta\|_{\infty}) \left(\int_{\Omega}(1+|Du|^2)^{\frac{p}{2}}\; dx\right)^{\tilde{q}},
\end{split}
\end{equation*}
with some power $\tilde{q}$ whose value depends on $p,q$ and $\gamma$.
>From this inequality the statement i) of Theorem~\ref{t:main} follows directly.

Now, we  check the validity of \eqref{asss}, which, by using of definition of $\theta$,
  it can be written as
\[q<2p\left(1-\frac{1}{2^*}\right).\]
If $n=2$, we can choose $2^*$ arbitrarily large, therefore in this case the condition \eqref{asss} reduces to $q<2p$, which is
exactly the assumption \eqref{ipotesip-q} for $n=2$. If $n\ge 3$ we have $2^*=2n/(n-2)$ and the above condition is then equivalent to
\[ q<p\frac{n+2}{n},\]
which is nothing else than the assumption \eqref{ipotesip-q}. Hence the proof of the statement i) is finished.

 \subsection{The case $q<p+2$ and $p<n$}
We again start to estimate the integral on the right hand side of \eqref{split}.
Using a simple inequality and the integration
by parts, we find that (here $K$ is again a generic constant depending only on $q$)
\begin{align}\nonumber
\int_{\Omega} &(1+|Du|^2)^{\frac{q}{2}} |D\tau|^2\; dx \le K+
\sum_{k=1}^n\sum_{\alpha=1}^N \int_{\Omega}(1+|Du|^2)^{\frac{q-2}{2}}
u^{\alpha}_{x_k}u^{\alpha}_{x_k}|D\tau|^2\,dx\\ \nonumber
&=  K-\sum_{k=1}^n\sum_{\alpha=1}^N \int_{\Omega}\left((1+|Du|^2)^{\frac{q-2}{2}}
u^{\alpha}_{x_k}|D\tau|^2 \right)_{x_k}u^{\alpha}\; dx\\
&\le K+K\|u\|_{\infty} \int_{\Omega}\left((1+|Du|^2)^{\frac{q-2}{2}}
|D^2 u||D\tau|^2 +(1+|Du|^2)^{\frac{q-1}{2}}|D\tau||D^2\tau|\right)\; dx.
\label{e:stima}\end{align}

Let us now set  $\tau:=\eta^{\gamma}$, $\gamma\ge 2$ to be chosen later, where $\eta\in \mathcal{C}^{\infty}_c(\Omega)$ is an arbitrary nonnegative cut-off function.

By  the Young 's inequality,
\begin{align*}\nonumber  &
K\|u\|_{\infty} \int_{\Omega}(1+|Du|^2)^{\frac{q-2}{2}}
|D^2 u||D\tau|^2\,dx
\\&=
\int_{ \Omega}
\left\{(1+|Du|^2)^{\frac{p-2}{4}}|D^2u|\tau
\right\}\left\{\|u\|_{\infty} (1+|Du|^2)^{\frac{2q-p-2}{4}}\frac{|D\tau|^2}{\tau}
\right\}
\,dx
\nonumber
\\&\le \varepsilon
\int_{\Omega}(1+|Du|^2)^{\frac{p-2}{2}}|D^2u|^2\tau^2\; dx
+ c_{\varepsilon,K} \|u\|^2_{\infty} \int_{\Omega}(1+|Du|^2)^{\frac{2q-p-2}{2}}\frac{|D\tau|^4}{\tau^2}\; dx.
\end{align*}
Therefore,
\begin{equation}\label{inter}
\begin{split}
&\int_{\Omega} (1+|Du|^2)^{\frac{q}{2}} |D\tau|^2\; dx \le \varepsilon \int_{\Omega}(1+|Du|^2)^{\frac{p-2}{2}}|D^2u|^2\tau^2\; dx \\
&\quad + K+K\|u\|^2_{\infty} \int_{\Omega}(1+|Du|^2)^{\frac{2q-p-2}{2}}\frac{|D\tau|^4}{\tau^2}\; dx +
K\|u\|_{\infty}\int_{\Omega}(1+|Du|^2)^{\frac{q-1}{2}}|D\tau||D^2\tau|\; dx,
\end{split}
\end{equation}
with a possibly different positive constant $K$ than before.

\medbreak
Let us now discuss first the case $q\in [p,p+1]$.

If $q$ belongs to this range, the above inequality immediately reduces to
\begin{equation*}
\begin{split}
&\int_{\Omega} (1+|Du|^2)^{\frac{q}{2}} |D\tau|^2\; dx \le \varepsilon \int_{\Omega}(1+|Du|^2)^{\frac{p-2}{2}}|D^2u|^2\tau^2\; dx \\
&\quad + K+K\|u\|^2_{\infty} \int_{\Omega}(1+|Du|^2)^{\frac{p}{2}}\frac{|D\tau|^4}{\tau^2}\; dx +K\|u\|_{\infty}\int_{\Omega}(1+|Du|^2)^{\frac{p}{2}}|D\tau||D^2\tau|\; dx.
\end{split}
\end{equation*}
 Let us now choose $\gamma=2$, that is  $\tau:=\eta^2$. Thus  we get
\begin{equation}\label{intera}
\begin{split}
&\int_{\Omega} (1+|Du|^2)^{\frac{q}{2}} |D\tau|^2\; dx \le \varepsilon \int_{\Omega}(1+|Du|^2)^{\frac{p-2}{2}}|D^2u|^2\tau^2\; dx \\
&\quad + K+C(\|u\|_{\infty}, \|\eta\|_{2,\infty})  \int_{\Omega}(1+|Du|^2)^{\frac{p}{2}}\; dx.
\end{split}
\end{equation}
Hence by \eqref{split} and taking a proper $\varepsilon>0$, so that we can absorb the first term on the right hand side in \eqref{intera} by the left hand side in \eqref{split}, it is not difficult to arrive to the statement ii) of Theorem~\ref{t:main} for $q\in [p,p+1]$.

\medbreak
Next, we focus on the case when $q\in (p+1,p+2)$.

There exist  $\theta_1,\theta_2\in (0,1)$ such that
\begin{align}
q-1&=p(1-\theta_1)+q\theta_1, && \theta_1:= \frac{q-1-p}{q-p},\\
2q-p-2&=p(1-\theta_2)+ q\theta_2, && \theta_2:=\frac{2(q-1-p)}{q-p}.
\end{align}
 In addition, considering $\tau=\eta^\gamma$ with
\begin{equation}\label{special}
  \gamma= \frac{2-\theta_2}{1-\theta_2},
\end{equation}
 we have that
$\frac{|D\tau|^{2-\theta_2}}{\tau}=\gamma |D\eta|^{2-\theta_2}$.

With this setting, we can now estimate the remaining integrals on the right hand side of \eqref{inter} by means of the H\"{o}lder inequality as follows
$$
\begin{aligned}
&\int_{\Omega}(1+|Du|^2)^{\frac{2q-p-2}{2}}\frac{|D\tau|^4}{\tau^2}\; dx = \int_{\Omega}\left((1+|Du|^2)^{\frac{q}{2}}|D\tau|^2\right)^{\theta_2}(1+|Du|^2)^{\frac{p(1-\theta_2)}{2}}\frac{|D\tau|^{4-2\theta_2}}{\tau^2}\; dx\\
&\quad\le  C\left\|\frac{|D\tau|^{4-2\theta_2}}{\tau^2}\right\|_{\infty}\|(1+|Du|)\|_p^{p(1-\theta_2)} \left(\int_{\Omega}(1+|Du|^2)^{\frac{q}{2}}|D\tau|^2\; dx\right)^{\theta_2}.
\end{aligned}
$$
Then the above estimate reduces to
\begin{equation}\label{godot}
\begin{aligned}
&\int_{\Omega}(1+|Du|^2)^{\frac{2q-p-2}{2}}\frac{|D\tau|^4}{\tau^2}\; dx \\
&\qquad \le  C(\theta_2,\|\eta\|_{1,\infty})\|(1+|Du|)\|_p^{p(1-\theta_2)} \left(\int_{\Omega}(1+|Du|^2)^{\frac{q}{2}}|D\tau|^2\; dx\right)^{\theta_2}.
\end{aligned}
\end{equation}

We proceed similarly also with the remaining integral in \eqref{inter}, i.e., using the H\"{o}lder inequality, we have
\begin{equation}\label{godot2}
\begin{aligned}
&\int_{\Omega}(1+|Du|^2)^{\frac{q-1}{2}}|D\tau||D^2\tau|\; dx\\ &\quad =\int_{\Omega}(1+|Du|^2)^{\frac{p(1-\theta_1)}{2}}\left((1+|Du|^2)^{\frac{q}{2}}|D\tau|^2\right)^{\theta_1}|D\tau|^{1-2\theta_1}|D^2\tau|\; dx\\
&\quad\le K\|(1+|Du|)\|_p^{p(1-\theta_1)} \||D\tau|^{1-2\theta_1}|D^2\tau|\|_{\infty} \left(\int_{\Omega}(1+|Du|^2)^{\frac{q}{2}}|D\tau|^2\; dx \right)^{\theta_1}\\
&\quad\le C(\|\tau\|_{2,\infty},\theta_2)\|(1+|Du|)\|_p^{p(1-\theta_1)} \left(\int_{\Omega}(1+|Du|^2)^{\frac{q}{2}}|D\tau|^2\; dx \right)^{\theta_1},
\end{aligned}
\end{equation}
where the last inequality follows from the fact that $1-2\theta_1=1-\theta_2>0$.

Finally, using \eqref{godot} and \eqref{godot2} in \eqref{inter}, keeping in mind the special choice of $\tau$ in \eqref{special} and applying the Young's  inequality (notice that $\theta_1,\theta_2<1$) we observe that
\begin{equation*}
\begin{split}
\int_{\Omega} (1+|Du|^2)^{\frac{q}{2}} |D\tau|^2\; dx &\le \varepsilon \int_{\Omega}(1+|Du|^2)^{\frac{p-2}{2}}|D^2u|^2\tau^2\; dx \\
&\quad + C(\varepsilon,\theta_2,\|\eta\|_{2,\infty}, \|u\|_{\infty}, \|u\|_{1,p}).
\end{split}
\end{equation*}
Thus, going back to \eqref{split}, choosing $\varepsilon>0$ sufficiently small to absorb the term involving the second derivatives by the left hand side, we finally get the statement ii) of Theorem~\ref{t:main}.

\section{Proof of Theorem~\ref{t:main2}}
%
%
%

In this final section we establish the existence of a weak solution to the
Dirichlet problem \eqref{dirichlet}, under the assumption \eqref{ipotesi-dato} on the boundary datum $u_0$; i.e.,
\[u_0 \in W^{1,r} (\Omega;\mathbb{R}^N), \qquad r:=\max\left\{2,p\frac{q-1}{p-1}\right\}.\]

 We use an approximation procedure.
 For arbitrary $\epsilon \in (0,1)$ we introduce the approximate problem ($\alpha=1,\ldots,N$)
\begin{equation}\label{systemepsilon}
\left\{\begin{array}{ll}\sum_{i=1}^n\frac{\partial }{\partial x_i}\left(A_{\epsilon,i}^{\alpha}(Du_{\epsilon})\right)=0
&\textrm{in } \Omega,\\
u_{\epsilon}=u_0 &\textrm{on }\partial \Omega,
\end{array}\right.
\end{equation}
where
$A_{\epsilon,i}^{\alpha}:  \mathbb{R}^{nN}\to
\mathbb{R}$  is defined as
\begin{equation}
\label{Aepsilon}A_{\epsilon,i}^{\alpha}(\xi):=A_{i}^{\alpha}(\xi)+\epsilon (1+|\xi|^2)^{\frac{\max\{q,2\}-2}{2}}\xi_i^{\alpha}.
\end{equation}

In addition, in case we deal with the statement ii) of the theorem, we shall require that
$u_0\in L^{\infty}(\partial \Omega)$.
%

Due to \eqref{crescitaai}, \eqref{casop>2} and \eqref{casop<2} we have that
$A_{\epsilon,i}^{\alpha}(\xi)$ satisfies the following properties:
\[\sum_{i=1}^n\sum_{\alpha=1}^N A_{\epsilon,i}^{\alpha}(\xi)\xi_i^{\alpha}\ge \epsilon |\xi|^{\max\{q,2\}}-\lambda,\]
\begin{equation}
\label{e:crescita-eps}
|A_{\epsilon,i}^{\alpha}(\xi)|\le M'(1+|\xi|)^{\max\{q,2\}-1}.
\end{equation}
for some positive $\lambda$ and $M'$  independent on $\epsilon$.
We can apply  the theory of monotone operators (see e.g.  \cite{lerlio65,browder,hart-stamp}) to prove the existence of a unique
solution  to \eqref{systemepsilon}, i.e., the existence of $u_{\epsilon}\in u_0+W^{1,\max\{q,2\}}(\Omega; \mathbb{R}^N)$ fulfilling
\begin{equation}
\int_{\Omega}\sum_{\alpha=1}^N \sum_{i=1}^n{A_{\epsilon,i}^{\alpha}}(Du_{\epsilon})\varphi_{x_i}^\alpha\,dx=0\qquad \forall
\varphi\in W_0^{1,\max\{q,2\}}(\Omega;\mathbb{R}^N).
  \label{firstvariation}
\end{equation}

\subsection{First a~priori estimates}
We now derive  estimates for $u_{\epsilon}$  independent of $\epsilon$.

Using  $\varphi:=u_{\epsilon}-u_0$ as a test function in \eqref{firstvariation}, we get
 \begin{equation}\label{e:begin}
\begin{split}
0&=\int_{\Omega}\sum_{\alpha=1}^N \sum_{i=1}^n{A_{\epsilon,i}^{\alpha}}(Du_{\epsilon})((u_{\epsilon})_{x_i}^{\alpha}-(u_0)_{x_i}^{\alpha})\; dx\\
&=\int_{\Omega}\sum_{\alpha=1}^N \sum_{i=1}^n\left\{{A_{i}^{\alpha}}(Du_{\epsilon})((u_{\epsilon})_{x_i}^{\alpha}-(u_0)_{x_i}^{\alpha})+
\epsilon(1+|Du_{\epsilon}|^2)^{\frac{\max\{q,2\}-2}{2}} (u_{\epsilon})^{\alpha}_{x_i}((u_{\epsilon})_{x_i}^{\alpha}-(u_0)_{x_i}^{\alpha})\right\}\; dx\\
&\overset{\eqref{dis-ellitticita}}\ge \int_{\Omega}\left(K^{-1}|Du_{\epsilon}|^p -(1+|Du_0|^2)^{\frac{p(q-1)}{2(p-1)}} +
\epsilon(1+|Du_{\epsilon}|^2)^{\frac{\max\{q,2\}-2}{2}}|Du_{\epsilon}|(|Du_{\epsilon}|-|Du_0|)\right)\; dx.
\end{split}
\end{equation}

Since
\begin{equation*}
\begin{split}(1+|Du_{\epsilon}|^2)^{\frac{\max\{q,2\}-2}{2}}|Du_{\epsilon}|(|Du_{\epsilon}|-|Du_0|)=&
(1+|Du_{\epsilon}|^2)^{\frac{\max\{q,2\}-2}{2}}|Du_{\epsilon}|^2\\ &-(1+|Du_{\epsilon}|^2)^{\frac{\max\{q,2\}-2}{2}}|Du_{\epsilon}||Du_0|,\end{split}
\end{equation*}
then \eqref{e:begin} implies
\begin{align}\nonumber &\int_{\Omega}\left(|Du_{\epsilon}|^p  +\epsilon
(1+|Du_{\epsilon}|^2)^{\frac{\max\{q,2\}-2}{2}}|Du_{\epsilon}|^2\right)\; dx\\
&\le c
\int_{\Omega}\left((1+|Du_0|^2)^{\frac{p(q-1)}{2(p-1)}}+\epsilon(1+|Du_{\epsilon}|^2)^{\frac{\max\{q,2\}-2}{2}}|Du_{\epsilon}||Du_0|\right) \,dx.
\label{e:begin2}\end{align}
We claim that \eqref{e:begin2} implies
\begin{equation} \int_{\Omega}\left(|Du_{\epsilon}|^p  + \frac{\epsilon}{2}
(1+|Du_{\epsilon}|^2)^{\frac{\max\{2,q\}-2}{2}}|Du_{\epsilon}|^2\right)\; dx\le c
\int_{\Omega}(1+|Du_0|^2)^{\frac{r}{2}}\,dx
\label{goal}
\end{equation}

 If $q\le 2$, we can conclude using Young's inequality with exponent $\frac12$ on the last  term in \eqref{e:begin2}:
\begin{align*}&|Du_{\epsilon}||Du_0|\le \frac{1}{2c} |Du_{\epsilon}|^2+c' |Du_0|^2.
\end{align*}
Therefore, recalling that
 $r=\max\{2,\frac{p(q-1)}{p-1}\}$ the inequality \eqref{goal} follows.

Otherwise, if  $q>2$, the last term in \eqref{e:begin2} can be estimate as follows:
\begin{equation}
\epsilon(1+|Du_{\epsilon}|^2)^{\frac{\max\{q,2\}-2}{2}}|Du_{\epsilon}||Du_0|
\le  \epsilon\left\{c(1+|Du_0|^2)^{\frac{r}{2}}+c(1+|Du_{\epsilon}|^2)^{\frac{q-2}{4}} |Du_{\epsilon}|^{\frac{q}{2}}
|Du_0|\right\}.\label{e:second}
 \end{equation}
Indeed, in $\{|Du_{\epsilon}|\le 1\}$ we have  \[(1+|Du_{\epsilon}|^2)^{\frac{\max\{q,2\}-2}{2}}|Du_{\epsilon}||Du_0|\le 2^{\frac{q-2}{2}}|Du_0|\le c(1+|Du_0|^2)^{\frac{r}{2}}\]
and, in $\{|Du_{\epsilon}|> 1\}$,
 \[(1+|Du_{\epsilon}|^2)^{\frac{\max\{q,2\}-2}{2}}|Du_{\epsilon}||Du_0|\le
 2^{\frac{q-2}{4}}(1+|Du_{\epsilon}|^2)^{\frac{q-2}{4}}|Du_{\epsilon}|^{\frac{q}{2}}|Du_0|\]
and \eqref{e:second} follows.

To estimate the last term in \eqref{e:second}, we use   Young's inequality with exponents $2, q ,\frac{2q}{q-2}$. Recalling that $\epsilon<1$, we have
\begin{align}\nonumber & \epsilon c(1+|Du_{\epsilon}|^2)^{\frac{q-2}{4}} |Du_{\epsilon}|^{\frac{q}{2}}
|Du_0|=\epsilon c\left\{(1+|Du_{\epsilon}|^2)^{\frac{q-2}{4}} |Du_{\epsilon}|\right\} |Du_{\epsilon}|^{\frac{q-2}{2}}
|Du_0|\\ &\le\nonumber  \frac{\epsilon}{8}
(1+|Du_{\epsilon}|^2)^{\frac{q-2}{2}}|Du_{\epsilon}|^2+
\frac{\epsilon}{8}|Du_{\epsilon}|^{q}
+
\epsilon c|Du_0|^q
\\ &\le
\frac{\epsilon}{4}
(1+|Du_{\epsilon}|^2)^{\frac{\max\{q,2\}-2}{2}}|Du_{\epsilon}|^2
+
c(1+|Du_0|^2)^{\frac{r}{2}}
\label{e:fourth}\end{align}
with $c$ independent of $\epsilon$.
Therefore, collecting
\eqref{e:begin2}, \eqref{e:second}  and \eqref{e:fourth}, the inequality \eqref{goal} follows also in the case
$q>2$.

Thus,   we can find a universal constant $C>0$ such that (using also the Poincar\'{e} inequality)
\begin{equation}
\|u_{\epsilon}\|_{1,p} + \epsilon \|u_{\epsilon}\|^{\max\{q,2\}}_{1,\max\{q,2\}}\le C. \label{f-ap}
\end{equation}

If the assumption \eqref{structure} holds, then for every $\alpha\in \{1,\ldots, N\}$ we have
\begin{equation}
\label{BulHLP}
\sum_{i=1}^nA_{\epsilon,i}^{\alpha}(\xi)\xi_i^{\alpha}\ge \epsilon (1+|\xi|^{2})^{\frac{\max\{2,q\}-2}{2}}|\xi^{\alpha}|^2\ge  \epsilon |\xi^{\alpha}|^{\max\{q,2\}}
\end{equation}
and \[|A_{\epsilon,i}^{\alpha}(\xi)|\le (K+1)(1+|\xi|^{2})^{\frac{\max\{2,q\}-1}{2}},\]
where $K$ is as in \eqref{crescitaai}. Next we denote $\tilde{M}:=\|u_0\|_{L^{\infty}(\partial \Omega)}$ and define
$$
\varphi^{\alpha}:=\max \{u_{\epsilon}^{\alpha}-\tilde{M},0\} \qquad \alpha\in \{1,\ldots, N\}.
$$
Evidently, $\varphi=(\varphi^1,\ldots,\varphi^N) \in W^{1,\max \{2,q\}}_0(\Omega;\mathbb{R}^N)$ and
can be used as a test function in \eqref{firstvariation}. Doing so, and using the definition of $\varphi$ we obtain
(here $\chi_{u_{\epsilon}^{\alpha}\ge \tilde{M}}$ denotes the characteristic function of the set, where $u^{\alpha}_{\epsilon}\ge \tilde{M}$)
\begin{equation}
0=\int_{\Omega}\sum_{\alpha=1}^N \sum_{i=1}^n{A_{\epsilon,i}^{\alpha}}(Du_{\epsilon})\varphi_{x_i}^\alpha\,dx=\int_{\Omega}\sum_{\alpha=1}^N
\sum_{i=1}^n{A_{\epsilon,i}^{\alpha}}(Du_{\epsilon})Du_{\varepsilon}^{\alpha}\chi_{u_{\epsilon}^{\alpha}\ge \tilde{M}}\,dx
  \label{firstvariationM}
\end{equation}
Using finally \eqref{BulHLP}, we see that
\begin{equation}
\begin{split}
0&=\int_{\Omega}\sum_{\alpha=1}^N \sum_{i=1}^n{A_{\epsilon,i}^{\alpha}}(Du_{\epsilon})Du_{\varepsilon}^{\alpha}\chi_{u_{\epsilon}^{\alpha}\ge \tilde{M}}\,dx\ge \epsilon
\int_{\Omega}\sum_{\alpha=1}^N |Du_{\epsilon}^{\alpha}|^{\max\{q,2\}}\chi_{u_{\epsilon}^{\alpha}\ge \tilde{M}}\,dx\\
&=\epsilon\int_{\Omega}\sum_{\alpha=1}^N |D\varphi^{\alpha}|^{\max\{q,2\}}\,dx.
  \label{firstvariationMM}
\end{split}
\end{equation}
Consequently, $\varphi$ is a constant function. Since it has zero trace, it must be identically zero and it directly follows from its definition that $u_{\varepsilon}^{\alpha}\le \tilde{M}=\|u_0\|_{L^{\infty}(\partial \Omega)}$ for all $\alpha\in \{1,\ldots, N\}$. The minimum principle can be obtained by repeating step by step the above procedure for a test function defined as
$$
\varphi^{\alpha}:=\min \{u_{\epsilon}^{\alpha}+\tilde{M},0\} \qquad \alpha\in \{1,\ldots, N\}.
$$
Therefore, we conclude that, for every $\epsilon\in (0,1)$,
\begin{equation}\label{apest3}
\|u_{\epsilon}\|_{L^{\infty}(\Omega)}\le \|u_0\|_{L^{\infty}(\partial \Omega)}.
\end{equation}

\subsection{Uniform higher order estimates}

Due to the proof of a~priori estimates we can use Theorem~\ref{T:main-weak} to get the existence of the second order   derivatives of $u_{\epsilon}$, but  with
their estimates depending on $\epsilon$. Nevertheless, we can repeat step by step the estimates in Theorem~\ref{T:main-weak}
to get the following inequality
\begin{equation}\label{hot}
\begin{aligned}
&\int_{\Omega}(1+|Du_{\epsilon}|^2)^{\frac{p-2}{2}}|D^2u_{\epsilon}|^2\tau^2\; dx \le c\int_{\Omega}(1+|Du_{\epsilon}|^2)^{\frac{q}{2}}|D\tau^2|\; dx \\
&\qquad -c\epsilon\int_{\Omega}\sum_{i,k=1}^n\sum_{\alpha=1}^N
\left((1+|Du_{\epsilon}|^2)^{\frac{\max\{q,2\}-2}{2}}(u_{\epsilon})^{\alpha}_{x_i}\right)_{x_k} \left((u_{\epsilon})^{\alpha}_{x_k}\tau^2 \right)_{x_i}\; dx
\end{aligned}
\end{equation}
for every $\tau\in \mathcal{C}_c^{\infty}(\Omega)$.
 Thus, we need to bound uniformly the last integral. By a rather standard manipulation  and using the Young inequality, it is not difficult to check that
$$
\begin{aligned}
&\sum_{i,k=1}^n\sum_{\alpha=1}^N \left((1+|Du_{\epsilon}|^2)^{\frac{\max\{q,2\}-2}{2}}(u_{\epsilon})^{\alpha}_{x_i}\right)_{x_k} \left((u_{\epsilon})^{\alpha}_{x_k}\tau^2 \right)_{x_i}\\
&\qquad \ge (1+|Du_{\epsilon}|^2)^{\frac{\max\{q,2\}-2}{2}}|D^2u_{\epsilon}|^2\tau^2 - 2\max\{q,2\} (1+|Du_{\epsilon}|^2)^{\frac{\max\{q,2\}-2}{2}}|D^2u_{\epsilon}|\tau |Du_{\epsilon}| |D\tau|\\
&\qquad \ge -C(1+|Du_{\epsilon}|^2)^{\frac{\max\{q,2\}}{2}}|D\tau|^2
\end{aligned}
$$  with $C$ independent of $\epsilon$.
Substituting this into \eqref{hot}, we derive
\begin{equation}\label{hot2}
\begin{aligned}
&\int_{\Omega}(1+|Du_{\epsilon}|^2)^{\frac{p-2}{2}}|D^2u_{\epsilon}|^2\tau^2\; dx \le c\int_{\Omega}(1+|Du_{\epsilon}|^2)^{\frac{q}{2}}|D\tau|^2\; dx\\
&\qquad  +c\epsilon \int_{\Omega}(1+|Du_{\epsilon}|^2)^{\frac{\max\{q,2\}}{2}}|D\tau|^2\; dx\\
&\qquad \overset{\eqref{f-ap}}\le c +c\int_{\Omega}(1+|Du_{\epsilon}|^2)^{\frac{q}{2}}|D\tau|^2\; dx.
\end{aligned}
\end{equation}
Hence, we are in the same starting position as in the proof of Theorem~\ref{t:main} and due to uniform ($\epsilon$-independent) uniform bounds \eqref{f-ap} and \eqref{apest3}, we deduce that for arbitrary open $\Omega'\subset \overline{\Omega'} \subset \Omega$,
\begin{equation}
\int_{\Omega'}\left(|Du_{\epsilon}|^{q}+|D V(Du_{\epsilon})|^2+ (1+|Du_{\epsilon}|^2)^{\frac{p-2}{2}}|D^2u_{\epsilon}|^2\right)\, dx \le C(\Omega', u_0).
\label{exdopoH3-V2ns}
\end{equation}
Further, it is then not difficult to observe with the help of the H\"{o}lder inequality that
\begin{equation}
\int_{\Omega'}|D^2u_{\epsilon}|^{\min\{2,p\}}\, dx \le C(\Omega', u_0).
\label{exdopoH3-V2ns3}
\end{equation}

\subsection{Limit $\epsilon \to 0$} Using the uniform bounds \eqref{f-ap}, \eqref{exdopoH3-V2ns} and \eqref{exdopoH3-V2ns3}, the compact Sobolev embedding and the diagonal procedure,
we can find a subsequence, that we do not relabel, and it exists \[u\in (u_0+W^{1,p}(\Omega; \mathbb{R}^N))\cap W_{\rm loc}^{1,q}(\Omega; \mathbb{R}^N) \]
such that for arbitrary open $\Omega'\subset \overline{\Omega'} \subset \Omega$, we have
\begin{align}
u^{\epsilon}&\rightharpoonup u &&\textrm{weakly in }W^{1,p}(\Omega; \mathbb{R}^N),\label{co1}\\
u^{\epsilon}&\rightharpoonup u &&\textrm{weakly in }W^{1,q}(\Omega'; \mathbb{R}^N),\label{co2}\\
Du^{\epsilon}&\to Du &&\textrm{strongly in }L^p(\Omega; \mathbb{R}^N),\label{co3}\\
Du^{\epsilon}&\to Du &&\textrm{almost everywhere in }\Omega,\label{co4}\\
\epsilon (1+|Du^{\epsilon}|^2)^{\frac{\max(2,q)-2}{2}}Du^{\epsilon}&\to 0 &&\textrm{strongly in }L^1(\Omega'; \mathbb{R}^{nN}).\label{co5}
\end{align}
Having \eqref{co1}--\eqref{co5}, it is easy to let $\epsilon \to 0$ in
\eqref{firstvariation} with arbitrary $\varphi \in \mathcal{C}^{\infty}_c(\Omega; \mathbb{R}^N)$ to deduce \eqref{soldebole-Dirichlet} for the same class of $\varphi$'s. The density result then leads to the validity of \eqref{soldebole-Dirichlet} in the full generality. This finishes the proof.


\end{document}